\documentclass[11pt]{amsart}

\usepackage{amsmath,amssymb,mathtools}
\usepackage[colorlinks=true,citecolor=blue,linkcolor=blue,urlcolor=blue]{hyperref}

\newtheorem{theorem}{Theorem}[section]
\newtheorem{proposition}[theorem]{Proposition}
\newtheorem{lemma}[theorem]{Lemma}
\newtheorem{corollary}[theorem]{Corollary}
\theoremstyle{definition}
\newtheorem{definition}[theorem]{Definition}
\theoremstyle{remark}
\newtheorem{remark}[theorem]{Remark}
\newtheorem{question}[theorem]{Question}

\newcommand{\C}{\mathbb C}
\newcommand{\D}{\mathbb D}

\newcommand{\E}{\mathbb E}
\DeclareMathOperator{\Var}{Var}
\DeclarePairedDelimiter{\abs}{\lvert}{\rvert}
\DeclarePairedDelimiter{\norm}{\lVert}{\rVert}

\title[The Korenblum radius of the Fock space]
{The sharp radius in Korenblum's maximum principle for the Fock space}

\author{Frank Wikström}
\address{Centre for Mathematical Sciences, Lund University,
Box 118, SE-221 00 Lund, Sweden}
\email{frank.wikstrom@math.lth.se}

\subjclass[2020]{Primary 30H20; Secondary 30C80, 44A60}
\keywords{Fock space, Korenblum's maximum principle, domination,
Schwarz--Pick lemma, moment problem}

\begin{document}

\begin{abstract}
Korenblum's maximum principle states that if $\abs{f}\le\abs{g}$ near the
boundary, then $\norm{f}\le\norm{g}$.  The optimal size of the region of
domination has been studied in Bergman spaces since 1991 and in Fock spaces
since 2006, but only bounds have been obtained.  We determine the optimal size
for the Fock space $F^2$ of entire functions that are square integrable with
respect to $e^{-\abs{z}^2}\,dA(z)$: if $f, g\in F^2$ and
$\abs{f(z)}\le\abs{g(z)}$ for all $\abs{z}>1$, then $\norm{f}\le\norm{g}$.  The
radius $1$ is optimal and we determine all pairs for which equality holds.  To
our knowledge, this is the first space of Bergman or Fock type in which the
optimal radius in Korenblum's principle has been determined.

The proof combines a Schwarz--Pick estimate centred at infinity with a
moment-duality argument, and requires no numerical computation.  The same method
gives a moment criterion for weighted Fock-type spaces.  For every
non-increasing radial weight, the optimal radius equals the upper bound given by
the pairs $f\equiv c$, $g(z)=z$.  In particular, the upper bounds of Wee and Le
for such weighted Fock spaces are sharp in the Hilbert space case. We also show
that the optimal radius is $\sqrt{(\beta+1)/\alpha}$ for the weight
$\abs{z}^{2\beta}e^{-\alpha\abs{z}^2}$, whenever $-1<\beta\le5.44$.
\end{abstract}

\maketitle

\section{Introduction}

Let $\D$ be the unit disk and let $dA$ denote area measure on $\C$. Korenblum's
maximum principle \cite{Korenblum1991} asserts that there is an absolute
constant $c \in (0,1)$ such that, whenever $f$ and $g$ belong to the Bergman
space $A^2(\D)$ and $\abs{f(z)} \le \abs{g(z)}$ for $c < \abs{z} < 1$, it
follows that $\norm{f}_{A^2} \le \norm{g}_{A^2}$.  This was proved by
Hayman~\cite{Hayman1999}, and extended to $A^p(\D)$, $p\ge1$, by
Hinkkanen~\cite{Hinkkanen1999}. For $0<p<1$, it
fails~\cite{BozinKarapetrovic2018}.

Once the principle is known to hold, the natural problem is to find the optimal
constant, and this has proved difficult.  For $A^2(\D)$ the optimal constant
$c_2$ is still unknown after more than thirty years, and the best current bounds
are $0.4263 \le c_2 < 0.6779$~\cite{Wang2008,Wikstrom2026}.  The situation is
the same in the other spaces where the principle has been studied, among them
$A^p(\D)$, weighted Bergman spaces, mixed norm spaces~\cite{Karapetrovic2022},
and the Fock spaces with their weighted variants.  In none of them has the
optimal constant been determined; only lower or upper bounds are available, and
often only one of the two, see the survey~\cite{WeeLe2023survey}.  For instance,
Efraimidis, Llinares and Vukoti\'c~\cite{EfraimidisLlinaresVukotic2025} proved
the principle in $A^p_w(\D)$, $1\le p<\infty$, for every integrable radial
weight~$w$, and showed that the optimal radius is always strictly less than~$1$.
(In the Hardy spaces the question degenerates, since domination on any
non-tangentially dense set already implies the norm
inequality~\cite[Theorem~4.1]{WeeLe2023survey}.)

In the Fock space, the role of the annulus $\{ c < \abs{z} < 1\}$ is played by
the exterior of a disk.  Let
\[
    \norm{f}^2 = \frac{1}{\pi} \int_{\C} \abs{f(z)}^2 e^{-\abs{z}^2}\,dA(z),
\]
and define the \emph{Fock space} $F^2=\{f \text{ entire}: \norm{f}<\infty\}$. We
say that $R>0$ is \emph{admissible} if $\norm{f} \le \norm{g}$ whenever $f, g
\in F^2$ satisfy $\abs{f(z)} \le \abs{g(z)}$ for all $\abs{z} > R$, and we call
the supremum $\mathcal{K}_{F^2}$ of the admissible radii the \emph{Korenblum radius}
of $F^2$.  Schuster~\cite[Theorem~2]{Schuster2006} proved that $R=0.54$ is
admissible, and Wang~\cite[Theorem~2]{Wang2006} improved this to $R=0.7248$.
Zhu~\cite[Theorem~2.43]{ZhuFock} established the principle in the spaces $F^p_\alpha$ for $p
\ge1 $, while Hu and Lou~\cite{HuLou2019} showed that it fails for $ 0 < p < 1$.
On the other side, Wee and Le~\cite{WeeLe2020} observed that the pair $f\equiv
c$, $g(z)=z$ gives $\mathcal{K}_{F^2} \le 1$, and asked whether the radius~$1$ is
admissible \cite[Question~4.1]{WeeLe2020}.  The survey~\cite{WeeLe2023survey}
records the state of the art as $0.7248 \le \mathcal{K}_{F^2} \le 1$.

Our main result answers this question.  To our knowledge, it is the first
determination of the optimal radius in Korenblum's principle for a space of
Bergman or Fock type.  Moreover, the optimal radius is
attained, and we describe all extremal pairs.

\begin{theorem}\label{thm:main} Let $f, g\in F^2$ and suppose that $\abs{f(z)}
\le \abs{g(z)}$ for all $\abs{z} > 1$.  Then $\norm{f} \le \norm{g}$.  Moreover,
$\norm{f} = \norm{g}$ if and only if either $f = \eta g$ for a unimodular
constant~$\eta$, or
\begin{equation}\label{eq:extremal}
    f(z)=c\,(\eta+az),\qquad g(z)=c\,(z+\bar a\eta)
\end{equation}
for some $c\in\C$, $\abs{\eta}=1$ and $\abs{a}<1$.
\end{theorem}

Since no radius $R>1$ is admissible (Lemma~\ref{lem:trivial-bound}), the
Korenblum radius of $F^2$ is exactly~$1$, and it is attained.  By scaling,
the Korenblum radius of the space $F^2_\alpha$ with the measure
$\frac{\alpha}{\pi}e^{-\alpha\abs{z}^2}\,dA$ is $1/\sqrt\alpha$
(Corollary~\ref{cor:alpha}).

The extremal pairs~\eqref{eq:extremal} form a family of pairs of linear
polynomials.  It contains the pair $f\equiv1$, $g(z)=z$ of Wee and Le
($a=0$).  For every member, $\abs{f}=\abs{g}$ on the whole unit circle, so the
domination is tight on the entire boundary of $\{\abs{z}>1\}$.  For the Bergman
space no such description is available.

The proof has two ingredients.  The first is a pointwise estimate on circles: if
$\abs{f} \le \abs{g}$ for $\abs{z} > R$, then, apart from the trivial cases
$g\equiv0$ and $f/g$ a unimodular constant, $\omega = f/g$ is a holomorphic map
from the disk $\{ \abs{z} > R \} \cup \{ \infty \}$ of the Riemann sphere into
$\D$, and the Schwarz--Pick lemma, centred at $\omega(\infty)$, controls the
positive and negative parts of the sequence $\abs{a_k}^2 - \abs{b_k}^2$ of
differences of squared Taylor coefficients on every circle $\abs{z} = r > R$
(Lemma~\ref{lem:circle}).  The second is the moment-duality principle introduced
by the author in~\cite{Wikstrom2026} for the Bergman space: integrating the
circle estimates against a suitable measure turns them into the norm inequality.
The proof below is self-contained and does not depend on~\cite{Wikstrom2026}.

The first ingredient is what makes the Fock space tractable.  The exterior of a
disk contains the point~$\infty$, and the Schwarz--Pick lemma centred at
$\omega(\infty)$ gives estimates that are sharp for Möbius maps, that is, for
the extremal pairs~\eqref{eq:extremal}.  An annulus has no distinguished point
of this kind.  There, the available estimates only control the
pseudo-hyperbolic diameter of the image of each circle, through the Möbius
pseudodistance of the annulus~\cite{Schuster2006,Wang2025}, and the resulting
bounds for~$c_2$ are not sharp.

For the Fock space, the measure can be written down explicitly and it turns out
to be the law of $1+E$, where $E$ is a standard exponential random variable.
The same argument works for any radial weight, and leads to the following
criterion.  For a radial weight $W$ on $[0,\infty)$ we write $m_k=\int_0^\infty
t^kW(t)\,dt$ and let $F^2_W$ be the space of entire functions with
$\norm{f}_W^2=\frac1\pi\int_\C\abs{f(z)}^2W(\abs{z}^2)\,dA(z)<\infty$; see
Section~\ref{sec:prelim} for the precise definitions.

\begin{theorem}\label{thm:criterion} Let $W$ be a radial weight with $m_0=1$ and
let $R>0$.  Suppose that there is a Borel probability measure $\nu$ supported on
$[R^2,\infty)$ such that for every $k\ge1$,
\begin{equation}\label{eq:moment-conditions}
    \int t^k\,d\nu(t)\ \ge\ m_k
    \qquad\text{and}\qquad
    R^2\int t^{k-1}\,d\nu(t)\ \le\ m_k .
\end{equation}
Then $R$ is admissible for $F^2_W$.
\end{theorem}

The condition $R^2\le m_1$ is necessary for admissibility
(Lemma~\ref{lem:trivial-bound}).  For many weights it is also sufficient.

\begin{theorem}\label{thm:decreasing}
Let $W$ be a radial weight that is non-increasing on $(0,\infty)$.  Then the
Korenblum radius of $F^2_W$ equals $\sqrt{m_1/m_0}$, and this radius is
admissible.
\end{theorem}

Theorem~\ref{thm:decreasing} applies, for instance, to the weights
$\exp(-\abs{z}^m)$ for every $m>0$, for which $W(t)=e^{-t^{m/2}}$,
$m_k=\frac2m\Gamma\bigl(\frac{2k+2}m\bigr)$ and hence
$\mathcal{K}_W=\sqrt{\Gamma(4/m)/\Gamma(2/m)}$, and to
$\abs{z}^{2\beta}\exp(-\alpha\abs{z}^2)$ with $-1<\beta\le0$.  It also shows
that for $p=2$ the upper bounds obtained by Wee and Le~\cite{WeeLe2023Fock} for
several families of weighted Fock spaces are sharp and attained, which answers
two of their open questions in this case (Remark~\ref{rem:weele}).  For the
weights $\abs{z}^{2\beta}\exp(-\alpha\abs{z}^2)$ we can go further.

\begin{corollary}\label{cor:gamma}
Let $\alpha>0$, and let $F^2_{\alpha,\beta}$ be the space of entire functions
with
\[
    \int_\C\abs{f(z)}^2\abs{z}^{2\beta}e^{-\alpha\abs{z}^2}\,dA(z)<\infty.
\]
If $-1<\beta\le\kappa_1-1=5.4475\ldots$, where $\kappa_1$ is the unique positive
solution of $\Gamma(\kappa)=(\kappa/e)^\kappa$, then the Korenblum radius of
$F^2_{\alpha,\beta}$ equals $\sqrt{(\beta+1)/\alpha}$, and this radius is
admissible.  Since $\kappa_1>2\pi$, this includes all $-1<\beta\le2\pi-1$.
\end{corollary}

Theorem~\ref{thm:decreasing} and Corollary~\ref{cor:gamma} follow from
Proposition~\ref{prop:tail}.  Normalize $m_0=1$ and let $T$ have density $W$.
Then the law of $\max(T,m_1)$ satisfies the conditions of Theorem~\ref{thm:criterion} with
$R^2=m_1$ if and only if the single inequality
$m_1\,\E(T-m_1)_+\le\Var T$ holds.  For the weights of
Corollary~\ref{cor:gamma} this happens exactly when $\beta\le\kappa_1-1$.
Theorem~\ref{thm:criterion} with $R^2=m_1$ does not cover all radial weights.
For weights concentrated near a circle, no suitable measure exists
(Proposition~\ref{prop:thin} in Section~\ref{sec:numerics}), and the same
holds for the weight $\abs z^{98}e^{-\abs z^2}$ (Proposition~\ref{prop:gamma50}).
We do not know whether the Korenblum radius is smaller than $\sqrt{m_1/m_0}$
for such weights.  Pairs with $\abs f\le\abs g$ outside the circle
$\abs z=\sqrt{m_1/m_0}$ and $\abs f=\abs g$ on it never violate the principle
there, for any radial weight (Proposition~\ref{prop:blaschke} and the remark
following it).

\subsection*{Organization}
Section~\ref{sec:prelim} collects notation and two elementary identities.
Section~\ref{sec:circle} contains the Schwarz--Pick estimate at infinity, and
Section~\ref{sec:duality} the moment-duality argument, which proves
Theorem~\ref{thm:criterion}.  Theorem~\ref{thm:main} is proved in
Section~\ref{sec:fock}, and Theorem~\ref{thm:decreasing} and
Corollary~\ref{cor:gamma} in Section~\ref{sec:weights}.
Section~\ref{sec:remarks} contains further remarks and open problems,
including the limits of the criterion mentioned above.

\subsection*{Tool and computational resource disclosure}

Claude Opus~5.5 (Anthropic) was used as a research assistant for literature
search, numerical exploration of the moment problems, and drafting.  The proofs
of all theorems and propositions are free of numerical computation, except
Proposition~\ref{prop:gamma50}.  Its certificate was found by linear
programming and is printed in Section~\ref{sec:numerics}, and its proof
reduces to finitely many sign checks in exact rational arithmetic.  An exact
verifier, together with the further certificates mentioned in
Section~\ref{sec:numerics}, is archived at~\cite{WikstromZenodo}.  The author
has checked all arguments and takes full responsibility for the content.

\section{Preliminaries}\label{sec:prelim}

\subsection{Radial weights}
By a \emph{radial weight} we mean a measurable function $W \ge 0$ on
$[0,\infty)$, not almost everywhere zero, whose moments
\[
    m_k=\int_0^\infty t^k\,W(t)\,dt,\qquad k=0,1,2,\dots,
\]
are all finite.  For an entire function $f(z)=\sum_{k\ge0}a_kz^k$ let
\[
    M_2(r,f)^2=\frac{1}{2\pi}\int_{0}^{2\pi}\abs{f(re^{i\theta})}^2\,d\theta
    =\sum_{k\ge0}\abs{a_k}^2r^{2k}.
\]
Integrating in polar coordinates and substituting $t=r^2$, we find
\begin{equation}\label{eq:norm}
    \begin{split}
        \norm{f}_W^2
            &= \frac{1}{\pi}\int_\C\abs{f(z)}^2W(\abs{z}^2)\,dA(z) \\
            &= \int_{0}^{\infty} M_2(\sqrt t,f)^2\,W(t)\,dt
            = \sum_{k \ge 0} \abs{a_k}^2 m_k .
    \end{split}
\end{equation}
We write $F^2_W$ for the space of entire functions with $\norm{f}_W<\infty$.
Multiplying $W$ by a constant $c>0$ multiplies all norms by $\sqrt c$, so we may
and shall assume $m_0=1$ whenever convenient.  For $W(t)=e^{-t}$, we recover the
standard Fock space $F^2$ with $m_k=k!$.

\begin{definition}
Let $W$ be a radial weight.  A radius $R>0$ is \emph{admissible} for $F^2_W$ if
$\norm{f}_W\le\norm{g}_W$ for all $f,g\in F^2_W$ with $\abs{f(z)}\le\abs{g(z)}$
for $\abs{z}>R$.  The \emph{Korenblum radius} $\mathcal{K}_W$ is the supremum of the
admissible radii.
\end{definition}

If $R$ is admissible, so is every $R'<R$.  The following upper bound was
observed for the classical Fock spaces in~\cite{WeeLe2020}, and for a
variety of weighted Fock spaces in~\cite{WeeLe2023Fock}; see
Remark~\ref{rem:weele}.

\begin{lemma}\label{lem:trivial-bound}
    For every radial weight, $\mathcal{K}_W\le\sqrt{m_1/m_0}$.
\end{lemma}

\begin{proof}
    Let $R>\sqrt{m_1/m_0}$, $f\equiv R$ and $g(z)=z$.  Then
    $\abs{f(z)}\le\abs{g(z)}$ for $\abs{z}\ge R$, while
    $\norm{f}_W^2=R^2m_0>m_1=\norm{g}_W^2$ by~\eqref{eq:norm}.
\end{proof}

\subsection{Two identities}
For $a, w \in \D$ let
\[
    \delta(w,a) = \abs[\Big]{\frac{w-a}{1-\bar aw}}
\]
denote the pseudo-hyperbolic distance.  We shall use the classical identity
\begin{equation}\label{eq:pseudo}
    1-\delta(w,a)^2=\frac{(1-\abs{a}^2)(1-\abs{w}^2)}{\abs{1-\bar aw}^2},
\end{equation}
and the following elementary identity, a complex version of the one used by
Schuster \cite{Schuster2006} and Wang \cite{Wang2006,Wang2025}.

\begin{lemma}\label{lem:identity}
Let $\lambda\in\D$, and let $A, B$ be vectors in a complex inner product
space.  Then
\[
    (1-\abs{\lambda}^2)\bigl(\norm{B}^2-\norm{A}^2\bigr)
    = \norm{B-\lambda A}^2 - \norm{A-\bar\lambda B}^2 .
\]
\end{lemma}

\begin{proof}
Expanding, the cross terms are $-2\operatorname{Re}(\bar\lambda\langle B,A\rangle)$
and $-2\operatorname{Re}(\lambda\langle A,B\rangle)$, which coincide.  The
remaining terms give
$(1-\abs{\lambda}^2)\norm{B}^2-(1-\abs{\lambda}^2)\norm{A}^2$.
\end{proof}

We shall apply Lemma~\ref{lem:identity} both to complex numbers and to
elements of~$F^2$.

\section{A Schwarz--Pick estimate at infinity}\label{sec:circle}

Throughout this section, $f(z)=\sum a_kz^k$ and $g(z)=\sum b_kz^k$ are entire
functions.  Write
\[
    d_k=\abs{a_k}^2-\abs{b_k}^2,\qquad
    x_k=\max(d_k,0),\qquad
    y_k=\max(-d_k,0),
\]
and let
\[
    X(t)=\sum_{k\ge0} x_k t^k,\qquad
    Y(t)=\sum_{k\ge0} y_k t^k.
\]
Since $f$ and $g$ are entire, $X$ and $Y$ are entire functions of $t$, and
\begin{equation}\label{eq:M2-difference}
    M_2(r,g)^2 - M_2(r,f)^2 = Y(r^2) - X(r^2).
\end{equation}

\begin{lemma}\label{lem:circle}
Let $R>0$ and suppose that $\abs{f(z)} \le \abs{g(z)}$ for all $\abs{z} > R$.  Then
\begin{equation}\label{eq:circle}
    X(t) \le \frac{R^2}{t}\,Y(t),\qquad \text{for $t\ge R^2$.}
\end{equation}
More precisely, suppose that $g\not\equiv0$ and that $f/g$ is not a
unimodular constant.  Then $\omega=f/g$ extends to a holomorphic function on
$U=\{\abs{z}>R\}\cup\{\infty\}$ with $a=\omega(\infty)\in\D$, and for every
$r>R$,
\begin{align}
    (1-\abs{a}^2)\,Y(r^2)
    &\le M_2(r,\,g-\bar af)^2 \label{eq:step1}\\
    &\le (1-\abs{a}^2)\,\frac{r^2}{r^2-R^2}\,\bigl(Y(r^2)-X(r^2)\bigr).
    \label{eq:step2}
\end{align}
\end{lemma}

\begin{proof}
If $g\equiv0$ then $f\equiv0$, and if $f=\eta g$ with $\abs{\eta}=1$ then
$d_k=0$ for all $k$; in both cases $X=Y=0$.  Assume neither holds.

Near a zero of $g$ in $\{\abs{z}>R\}$ the quotient $f/g$ is bounded, so it
has a removable singularity there.  Hence $\omega=f/g$ is holomorphic on
$\{\abs{z}>R\}$ with $\abs{\omega}\le1$.  Being bounded near $\infty$, it
extends holomorphically to the disk $U$ of the Riemann sphere.  If
$\abs{\omega}$ attained the value $1$ in $U$, the maximum principle would
make $\omega$ a unimodular constant, which we have excluded.  Thus
$\omega(U)\subset\D$; put $a=\omega(\infty)$.

The function $\psi(u)=\omega(R/u)$ is a holomorphic self-map of $\D$ with
$\psi(0)=a$, so by the Schwarz--Pick lemma
\begin{equation}\label{eq:schwarz-pick}
    \delta\bigl(\omega(z),a\bigr)\le\frac{R}{\abs{z}},\qquad \text{for $\abs{z}>R$.}
\end{equation}

Let $S = \{ k \ge 0 : \abs{a_k} < \abs{b_k}\}$, so that $Y(t) = \sum_{k\in S}
(\abs{b_k}^2-\abs{a_k}^2) t^k$. Lemma~\ref{lem:identity} with $\lambda=\bar a$,
$A=a_k$, $B=b_k$ gives
\begin{equation}\label{eq:dropped}
    (1-\abs{a}^2)\,Y(r^2)
    = M_2(r,g-\bar af)^2
      -\sum_{k\in S} \abs{a_k-ab_k}^2 r^{2k}
      -\sum_{k\notin S} \abs{b_k-\bar a a_k}^2 r^{2k},
\end{equation}
because $M_2(r,g-\bar af)^2=\sum_{k\ge0}\abs{b_k-\bar aa_k}^2r^{2k}$.  This
proves~\eqref{eq:step1}.

On the circle $\abs{z}=r$ we have $g-\bar af=g\,(1-\bar a\omega)$,
so~\eqref{eq:pseudo} and~\eqref{eq:schwarz-pick} give
\begin{equation}\label{eq:pointwise}
    \abs{g-\bar af}^2
    = (1-\abs{a}^2)\,\frac{\abs{g}^2(1-\abs{\omega}^2)}{1-\delta(\omega,a)^2}
    \le (1-\abs{a}^2)\,\frac{r^2}{r^2-R^2}\,\bigl(\abs{g}^2-\abs{f}^2\bigr).
\end{equation}
Integrating over the circle and using~\eqref{eq:M2-difference}
gives~\eqref{eq:step2}.

Combining~\eqref{eq:step1} and~\eqref{eq:step2} and dividing by
$1-\abs{a}^2>0$, we get $Y(t)\le\frac{t}{t-R^2}\bigl(Y(t)-X(t)\bigr)$ for
$t>R^2$, which is~\eqref{eq:circle}.  The case $t=R^2$ follows by continuity.
\end{proof}

\section{Moment duality}\label{sec:duality}

\begin{proof}[Proof of Theorem~\ref{thm:criterion}]
Let $f,g\in F^2_W$ with $\abs{f(z)} \le \abs{g(z)}$ for $\abs{z} > R$, and let
$x_k, y_k, X, Y$ be as in Section~\ref{sec:circle}.  By~\eqref{eq:norm},
\[
    \norm{f}_W^2 - \norm{g}_W^2 = \sum_{k\ge0}(x_k-y_k)\,m_k,
\]
and $\sum_k y_k m_k \le \norm{g}_W^2 < \infty$.  Since $\nu$ is a probability
measure, the first condition in~\eqref{eq:moment-conditions} also holds for
$k=0$.  Using it, then~\eqref{eq:circle} on the support of $\nu$, and finally
the second condition in~\eqref{eq:moment-conditions} together with $R^2/t \le 1$
on $[R^2,\infty)$, we obtain from Tonelli's theorem
\begin{equation}\label{eq:chain}
\begin{split}
    \sum_{k \ge 0} x_k m_k &\le \sum_{k \ge 0} x_k \int t^k\,d\nu(t)
        = \int X\,d\nu \le \int \frac{R^2}{t}\,Y(t)\,d\nu(t) \\
    &= y_0\int\frac{R^2}{t}\,d\nu(t) + \sum_{k \ge 1} y_k\,R^2 \!\int t^{k-1}\,d\nu(t)
     \le\sum_{k \ge 0} y_k m_k.
\end{split}
\end{equation}
Hence, $\norm{f}_W \le \norm{g}_W$.
\end{proof}

\begin{remark}
Theorem~\ref{thm:criterion} is what one obtains by averaging the circle
estimates~\eqref{eq:circle} against a probability measure on
$[R^2,\infty)$; it is a sufficient condition, and we do not claim a
converse.  The conditions~\eqref{eq:moment-conditions} say that the
moments of~$\nu$ lie above the moments of the weight, while the moments of~$\nu$
damped by the factor $R^2/t$ lie below them;
compare~\cite[Lemma~3.1]{Wikstrom2026}.  Note that~$\nu$ may have an atom
at~$R^2$.
\end{remark}

\section{The Fock space}\label{sec:fock}

For the Fock space, $W(t)=e^{-t}$, so $m_k=k!$.  Let $\nu$ be the probability
measure on $[1,\infty)$ with density $e^{-(t-1)}$, that is, the law of $1+E$
with $E$ a standard exponential random variable.

\begin{lemma}\label{lem:nu}
For every integer $k\ge0$,
\[
    \int t^k\,d\nu(t)=k!\sum_{i=0}^k\frac1{i!}.
\]
Consequently, $\int t^k\,d\nu\ge k!$ for $k\ge0$, and
$\int t^{k-1}\,d\nu\le k!$ for $k\ge1$.
\end{lemma}

\begin{proof}
By the binomial theorem and $\int_0^\infty s^je^{-s}\,ds=j!$,
\[
    \int_1^\infty t^ke^{-(t-1)}\,dt
    =\int_0^\infty(1+s)^ke^{-s}\,ds
    =\sum_{j=0}^k\binom kj j!
    =k!\sum_{i=0}^k\frac1{i!}.
\]
The first inequality is immediate.  For the second, we must show
$\sum_{i=0}^{k-1}1/i!\le k$.  Equality holds for $k=1,2$, and for $k\ge3$ the
left-hand side is less than $e<3\le k$.
\end{proof}

\begin{proof}[Proof of Theorem~\ref{thm:main}]
By Lemma~\ref{lem:nu}, $\nu$ satisfies~\eqref{eq:moment-conditions} with
$R=1$ and $m_k=k!$.  Theorem~\ref{thm:criterion} therefore gives
$\norm{f}\le\norm{g}$.

If $f=\eta g$ with $\abs{\eta}=1$, clearly $\norm{f}=\norm{g}$.  For the
pairs~\eqref{eq:extremal} we have, for $\abs{z}\ge1$,
\[
    \abs{z+\bar a\eta}^2 - \abs{\eta+az}^2 = (1-\abs{a}^2)(\abs{z}^2-1) \ge 0,
\]
so $\abs{f} \le \abs{g}$ there, and
$\norm{f}^2 = \abs{c}^2(1+\abs{a}^2) = \norm{g}^2$.

Conversely, suppose that $\norm{f} = \norm{g}$ and that $f$ is not a unimodular
multiple of $g$.  Then $g \not\equiv 0$.  If $\omega = f/g$ were a constant, it
would have modulus less than $1$, forcing $\norm{f} < \norm{g}$.  So $\omega$ is
non-constant, and Lemma~\ref{lem:circle} applies with $R=1$ and
$a=\omega(\infty)\in\D$.

The two ends of~\eqref{eq:chain} are $\sum_k x_k m_k$ and $\sum_k y_k m_k$,
whose difference is $\norm{f}^2 - \norm{g}^2 = 0$.  Hence, every inequality
in~\eqref{eq:chain} is an equality.  In particular,
$\int(Y(t)/t-X(t))\,d\nu(t)=0$ with a non-negative integrand.  Since $\nu$ has a
positive density on $(1,\infty)$ and $X, Y$ are continuous, $X(t)=Y(t)/t$ for
all $t>1$.  This is exactly the statement that equality holds in
$Y(t)\le\frac{t}{t-1}(Y(t)-X(t))$, so for every $r>1$ both~\eqref{eq:step1}
and~\eqref{eq:step2} are equalities.

Equality in~\eqref{eq:step2} means that the integrals over $\abs{z}=r$ of the
two sides of~\eqref{eq:pointwise} agree.  The function $g$ has only finitely
many zeros on each circle and $\abs{\omega} < 1$, so
$\abs{g}^2(1-\abs{\omega}^2) > 0$ almost everywhere on the circle.  Hence
$\delta(\omega(z),a)=1/r$ for almost every, and by continuity every, $z$ with
$\abs{z}=r$.  As $r>1$ is arbitrary, the function $\psi(u)=\omega(1/u)$
satisfies $\delta(\psi(u), \psi(0)) = \abs{u}$ for $0 < \abs{u} < 1$.  By the
equality case in Schwarz--Pick, $\psi$ is an automorphism of $\D$, i.e.\
$\psi(u)=(a+\eta u)/(1+\bar a\eta u)$ for some $\abs{\eta}=1$.  Therefore,
\[
    \omega(z)=\frac{\eta+az}{z+\bar a\eta},
\]
and $(z+\bar a\eta)f(z)=(\eta+az)g(z)$ on $\{\abs{z}>1\}$, hence on $\C$.  At
$z_0=-\bar a\eta$ the factor $\eta+az_0=\eta(1-\abs{a}^2)$ does not vanish,
so $g(z_0)=0$.  We can thus write
\[
    f(z)=(\eta+az)\,h(z),\qquad g(z)=(z+\bar a\eta)\,h(z)
\]
with $h$ entire.  Moreover, $h\in F^2$, since $\abs{h} \le 2\abs{g}$ for
$\abs{z-z_0} \ge 1/2$ and $h$ is bounded near $z_0$.

Now apply Lemma~\ref{lem:identity} in $F^2$ with $\lambda=\bar a$, $A=f$,
$B=g$.  We have $g-\bar af=(1-\abs{a}^2)\,zh$ and $f-ag=(1-\abs{a}^2)\,\eta h$;
in particular $zh=(g-\bar af)/(1-\abs{a}^2)$ belongs to $F^2$.  Hence
\[
    (1-\abs{a}^2)\bigl(\norm{g}^2-\norm{f}^2\bigr)
    =(1-\abs{a}^2)^2\bigl(\norm{zh}^2-\norm{h}^2\bigr).
\]
Writing $h(z)=\sum h_kz^k$, we have
$\norm{zh}^2-\norm{h}^2=\sum_k\abs{h_k}^2\bigl((k+1)!-k!\bigr) =\sum_k k\cdot
k!\,\abs{h_k}^2$.  Thus $\norm{f}=\norm{g}$ forces $h_k=0$ for $k\ge1$, i.e.\
$h\equiv c$ is constant, and $(f,g)$ is of the form~\eqref{eq:extremal}.
\end{proof}

\begin{corollary}\label{cor:alpha} Let $\alpha>0$ and let $F^2_\alpha$ be the
space of entire functions with
$\norm{f}_\alpha^2=\frac{\alpha}{\pi}\int_\C\abs{f}^2e^{-\alpha\abs{z}^2}\,dA<\infty$.
If $f,g\in F^2_\alpha$ and $\abs{f(z)}\le\abs{g(z)}$ for
$\abs{z}>1/\sqrt\alpha$, then $\norm{f}_\alpha\le\norm{g}_\alpha$.  The
Korenblum radius of $F^2_\alpha$ is $1/\sqrt\alpha$.
\end{corollary}

\begin{proof}
The map $f\mapsto f(\,\cdot\,/\sqrt\alpha)$ is an isometry of $F^2_\alpha$
onto $F^2$ which transforms the hypothesis into the hypothesis of
Theorem~\ref{thm:main}.  The upper bound is Lemma~\ref{lem:trivial-bound},
since $m_1/m_0=1/\alpha$.
\end{proof}

\begin{remark}
The measure $\nu$ is not unique.  For instance, the law of $\max(E,1)$ also
satisfies \eqref{eq:moment-conditions}; this is the case $\beta=0$ of
Section~\ref{sec:weights}.  The characterization of the extremal pairs,
however, only used that $\nu$ has full support in $[1,\infty)$.
\end{remark}

\section{Fock-type spaces with radial weights}\label{sec:weights}

In this section, $W$ is a radial weight with $m_0=1$, and it is convenient to
use probabilistic notation: $T$ denotes a random variable with density $W$, so
that $\E\,T^k=m_k$, $s=m_1=\E\,T$ and $\Var T=m_2-s^2$.  We write
$u_+=\max(u,0)$.

\begin{proposition}\label{prop:tail}
The law $\nu$ of $\max(T,s)$ satisfies \eqref{eq:moment-conditions} with
$R^2=s$ if and only if
\begin{equation}\label{eq:tail}
    s\,\E\,(T-s)_+\ \le\ \Var T .
\end{equation}
In that case $\mathcal{K}_W=\sqrt{m_1/m_0}$, and this radius is admissible.
\end{proposition}

\begin{proof}
Put $X=\max(T,s)$.  Since $X\ge T$, the first condition in
\eqref{eq:moment-conditions} holds for every $k$.  The second condition reads
$s\,\E\,X^{k-1}\le m_k$.  For $k=1$ it is the equality $s=m_1$.  Since
$\E\,T=s$, we have $\E\,(s-T)_+=\E\,(T-s)_+$ and hence
$\E\,X=s+\E\,(T-s)_+$.  So for $k=2$ the condition $s\,\E\,X\le m_2$ is
exactly \eqref{eq:tail}.

Assume \eqref{eq:tail}, and let $k=j+1$ with $j\ge2$.  For $x\ge0$ we have
$x^j=j(j-1)\int_0^\infty a^{j-2}(x-a)_+\,da$, so by Tonelli's theorem
\begin{align*}
    s\,\E\,X^j&=j(j-1)\int_0^\infty a^{j-2}\,s\,\E\,(X-a)_+\,da,\\
    m_{j+1}=\E\,[T\cdot T^j]&=j(j-1)\int_0^\infty a^{j-2}\,\E\,[T(T-a)_+]\,da .
\end{align*}
It therefore suffices to show that $s\,\E\,(X-a)_+\le\E\,[T(T-a)_+]$ for every
$a>0$.  If $a\ge s$, then $(X-a)_+=(T-a)_+$ and
$T(T-a)_+\ge a(T-a)_+\ge s(T-a)_+$.  If $0<a<s$, then $X-a>0$, and by the case
$k=2$,
\[
    s\,\E\,(X-a)_+=s\,\E\,X-sa\le m_2-sa=\E\,[T(T-a)]\le\E\,[T(T-a)_+] .
\]
The last assertion follows from Theorem~\ref{thm:criterion} and
Lemma~\ref{lem:trivial-bound}.
\end{proof}

\begin{remark}
Let $\hat T$ be the size-biased variable, with density $tW(t)/s$, so that
$\E\,\hat T^{\,j}=m_{j+1}/s$.  The second family of conditions in
\eqref{eq:moment-conditions} says that $\E\,X^j\le\E\,\hat T^{\,j}$ for all
$j\ge0$.  Since $\E\,[T(T-a)_+]=s\,\E\,(\hat T-a)_+$, the proof shows that $X$
is dominated by $\hat T$ in the increasing convex order
\cite[Section~4.A]{ShakedShanthikumar} as soon as $\E\,X\le\E\,\hat T$.  This is
why the single condition \eqref{eq:tail} controls all the moments.  The
condition is invariant under scaling of $T$.  It compares the mean excess of
$T$ over its mean with the variance of $T$.
\end{remark}

\begin{lemma}\label{lem:decreasing}
If $W$ is non-increasing on $(0,\infty)$, then
\[
    s\,\E\,(T-s)_+\le\tfrac34\Var T,
\]
with equality only if $T$ is uniformly distributed on $[0,2s]$.
\end{lemma}

\begin{proof}
Changing $W$ on a null set, we may assume that $W$ is right-continuous.  Since
$W$ is integrable and non-increasing, $W(t)\to0$ as $t\to\infty$.  Let
$\lambda$ be the measure on $(0,\infty)$ given by $d\lambda(b)=b\,d(-W)(b)$.
Then $W(t)=\int_{(t,\infty)}b^{-1}\,d\lambda(b)$ for $t>0$, and Tonelli's
theorem gives, for every measurable $g\ge0$,
\begin{equation}\label{eq:mixture}
    \E\,g(T)=\int_0^\infty g(t)W(t)\,dt
    =\int_{(0,\infty)}\Bigl(\frac1b\int_0^bg(t)\,dt\Bigr)\,d\lambda(b).
\end{equation}
With $g\equiv1$ this shows that $\lambda$ is a probability measure; let $B$
have law $\lambda$.  (Thus $T$ has the law of $UB$, where $U$ is uniform on
$[0,1]$ and independent of $B$.)  By \eqref{eq:mixture},
\[
    s=\tfrac12\E\,B,\qquad \E\,T^2=\tfrac13\E\,B^2,\qquad
    \E\,(T-s)_+=\E\,\frac{(B-s)_+^2}{2B}.
\]
We claim that for all $b>0$,
\begin{equation}\label{eq:pointwise-b}
    \frac{b^2}3-\frac{2s}3\cdot\frac{(b-s)_+^2}{b}
    \ \ge\ s^2+\frac{5s}6\,(b-2s),
\end{equation}
with equality only for $b=2s$.  Both sides are homogeneous of degree $2$ in
$(b,s)$, so we may take $s=1$.  For $0<b\le1$ the difference of the two sides
is $(2b^2-5b+4)/6>0$.  For $b>1$ the difference, multiplied by $6b$, equals
$2b^3-9b^2+12b-4=(b-2)^2(2b-1)$, which is positive except at $b=2$.  Taking
expectations in \eqref{eq:pointwise-b} and using $\E\,B=2s$, we get
$\E\,T^2-\frac43\,s\,\E\,(T-s)_+\ge s^2$, which is the assertion.  Equality
forces $B=2s$ almost surely, that is, $T$ is uniform on $[0,2s]$.
\end{proof}

\begin{proof}[Proof of Theorem~\ref{thm:decreasing}]
By Lemma~\ref{lem:decreasing}, condition \eqref{eq:tail} holds, and
Proposition~\ref{prop:tail} applies.
\end{proof}

\begin{remark}\label{rem:weele}
Wee and Le~\cite{WeeLe2023Fock} obtained upper bounds for the Korenblum
constants of several families of weighted Fock spaces, all from the pair
$f\equiv c$, $g(z)=z$.  For $p=2$ the norms of these spaces are, up to a
constant factor, of the form $\norm{\cdot}_W$ with
\[
    W(\abs{z}^2)=e^{-\alpha\abs{z}^m}\quad\text{and}\quad
    W(\abs{z}^2)=\sum_{n\ge1}\abs{d_n}\,e^{-\alpha\lambda_n\abs{z}^m},
\]
where $\alpha,m>0$, $0<\lambda_n\uparrow\infty$, $d_n\ne0$ and
\begin{equation}\label{eq:weele-summable}
    \sum_{n\ge1}\abs{d_n}\lambda_n^{-2/m}<\infty
\end{equation}
(\cite[Theorems~2.2, 3.2 and~3.6, Definition~3.11 and condition~(20)]{WeeLe2023Fock};
the infinite intersections there carry the norm
$\norm{f}^2=\sum_n\norm{f}_n^2$).  Condition~\eqref{eq:weele-summable} says
that $m_0<\infty$, and since $\lambda_n\ge\lambda_1$ it implies that all
moments are finite.  These weights are non-increasing, so by
Theorem~\ref{thm:decreasing} their upper bounds for $p=2$ are the exact
Korenblum radii, and these are attained.  Explicitly, since
$\int_0^\infty t^ke^{-ct^{m/2}}\,dt=\frac2m\,c^{-(2k+2)/m}\,\Gamma\bigl(\frac{2k+2}m\bigr)$,
the mixture weight has
\[
    \mathcal{K}_W^2=\frac{m_1}{m_0}
    =\alpha^{-2/m}\,\frac{\Gamma(4/m)}{\Gamma(2/m)}\,
    \frac{\sum_n\abs{d_n}\lambda_n^{-4/m}}{\sum_n\abs{d_n}\lambda_n^{-2/m}} .
\]
This answers \cite[Questions~(1) and~(2) in Section~5.2]{WeeLe2023Fock} for
$p=2$.  The ratio of the squared Korenblum radii of the infinite intersection
$F_{\{\lambda_n,d_n\}}$ and of $F_{\lambda_1,d_1}$ is the average of the
numbers $(\lambda_1/\lambda_n)^{2/m}\le1$ with the positive weights
$\abs{d_n}\lambda_n^{-2/m}$.  Since $\lambda_n\to\infty$, it is strictly less
than~$1$, so the first radius is strictly smaller.  For $\lambda_n=n^m$ and
$d_n=1$ the ratio of the sums is $\zeta(4)/\zeta(2)=\pi^2/15$, and therefore
\[
    \mathcal{K}_{F_{\{m,n^m,1\}}}=\frac{\pi}{\sqrt{15}}\,\mathcal{K}_{F_m}
\]
for all $\alpha,m>0$.
\end{remark}

\begin{proof}[Proof of Corollary~\ref{cor:gamma}]
By scaling we may assume $\alpha=1$.  After normalization,
$W(t)=t^\beta e^{-t}/\Gamma(\beta+1)$, so $T$ has the Gamma distribution
with shape $\kappa=\beta+1$, and $s=\Var T=\kappa$.  Write
$\Gamma(\kappa,x)=\int_x^\infty t^{\kappa-1}e^{-t}\,dt$.  Using
$\Gamma(\kappa+1,x)=\kappa\,\Gamma(\kappa,x)+x^\kappa e^{-x}$, one finds
\[
    \E\,(T-\kappa)_+
    =\frac{\Gamma(\kappa+1,\kappa)-\kappa\,\Gamma(\kappa,\kappa)}{\Gamma(\kappa)}
    =\frac{\kappa^\kappa e^{-\kappa}}{\Gamma(\kappa)} .
\]
Hence \eqref{eq:tail} is equivalent to $h(\kappa)\ge0$, where
$h(\kappa)=\log\Gamma(\kappa)-\kappa\log\kappa+\kappa$.  By Stirling's formula
with Binet's remainder \cite[Chapter~XII]{WhittakerWatson},
\begin{equation}\label{eq:stirling}
    h(\kappa)=\frac12\log\frac{2\pi}\kappa+\mu(\kappa),
    \qquad
    \mu(\kappa)=\int_0^\infty\Bigl(\frac12-\frac1t+\frac1{e^t-1}\Bigr)
    \frac{e^{-\kappa t}}t\,dt,
\end{equation}
and the integrand is positive.  Thus $\mu$ is positive and decreasing, and
$\mu(\kappa)\to0$ as $\kappa\to\infty$.  It follows that $h$ is strictly
decreasing, with $h(\kappa)\to+\infty$ as $\kappa\to0$ and
$h(\kappa)\to-\infty$ as $\kappa\to\infty$.  Hence $h$ has a unique zero
$\kappa_1$, and $h\ge0$ exactly on $(0,\kappa_1]$.  Since
$h(2\pi)=\mu(2\pi)>0$, we have $\kappa_1>2\pi$.  Now
Proposition~\ref{prop:tail} gives the result.
\end{proof}

\begin{remark}\label{rem:kappa1}
Numerically $\kappa_1=6.4475869\ldots$.  By \eqref{eq:stirling}, $\kappa_1$ is
the solution of $\kappa=2\pi e^{2\mu(\kappa)}$, so it exceeds $2\pi$ only by
the Stirling correction.  The number $2\pi$ has a simple meaning: for large
$\kappa$ the Gamma distribution is close to a normal distribution with mean
and variance $\kappa$, for which $\E\,(T-\kappa)_+\approx\sqrt{\kappa/(2\pi)}$,
while \eqref{eq:tail} asks for $\E\,(T-\kappa)_+\le1$.  For $\kappa>\kappa_1$
the law of $\max(T,s)$ violates \eqref{eq:moment-conditions} with $k=2$, so
$\kappa_1$ is the exact limit of this measure.  What happens for larger
$\kappa$ is discussed in Section~\ref{sec:numerics}.
\end{remark}

\begin{remark}
For a non-negative integer $\beta=m$, the space $F^2_{\alpha,m}$ is a
Fock--Sobolev space.  By a theorem of Cho and Zhu \cite{ChoZhu2012} (in one
variable, and after rescaling to our normalization), an entire function $f$
satisfies $f^{(m)}\in F^2_\alpha$ if and only if $z^mf\in F^2_\alpha$, that is,
if and only if $f\in F^2_{\alpha,m}$, with equivalent norms.  The Korenblum
radius depends on the norm and not only on the space, so
Corollary~\ref{cor:gamma} concerns the norm
$\bigl(\int\abs{f}^2\abs{z}^{2m}e^{-\alpha\abs{z}^2}\,dA\bigr)^{1/2}$ and not
the Fock--Sobolev norm.
\end{remark}

\section{Further remarks and open problems}\label{sec:remarks}

\subsection{Limits of the criterion}\label{sec:numerics}
In this subsection $m_0=1$ and $s=m_1$.  Theorem~\ref{thm:criterion} is only a
sufficient condition, so a weight for which no suitable measure exists may still
have Korenblum radius $\sqrt s$.  A natural place to look for counterexamples
is among pairs with $\abs f=\abs g$ on the circle $\abs z=\sqrt s$.  The
extremal pairs of Theorem~\ref{thm:main} are of this kind, and so are the
exterior analogues of Wang's pairs~\cite{Wang2008},
\[
    f(z)=R\,(R^N+az^N)\,Q(z^N),\qquad g(z)=z\,(aR^N+z^N)\,Q(z^N),
\]
with $0<a<1$ and $Q$ entire, for which $f(z)=B(R/z)\,g(z)$ with the Blaschke
product $B(u)=u(u^N+a)/(1+au^N)$.  The next proposition shows that such pairs
never violate the principle at the radius $\sqrt s$, for any weight.

\begin{proposition}\label{prop:blaschke}
Let $W$ be a radial weight with $m_0=1$, let $0<\rho\le\sqrt s$, and let $B$ be
a finite Blaschke product of degree $n\ge1$.  If $f,g\in F^2_W$, $g\not\equiv0$,
and $f(z)=B(\rho/z)\,g(z)$, then $\abs{f}\le\abs{g}$ for $\abs z\ge\rho$ and
$\norm f_W\le\norm g_W$.  Equality holds if and only if $\rho=\sqrt s$, $n=1$
and $g$ is a linear polynomial.
\end{proposition}

\begin{proof}
Write $B(u)=\eta\prod_{j=1}^n(u-\alpha_j)/(1-\bar\alpha_ju)$ with $\abs\eta=1$,
$\alpha_j\in\D$, and put $p_j(z)=\rho-\alpha_jz$, $q_j(z)=z-\bar\alpha_j\rho$,
so that $B(\rho/z)=\eta\prod_jp_j(z)/q_j(z)$.  The first claim holds because
$\abs{\rho/z}\le1$.  Since $p_i(\bar\alpha_j\rho)=\rho(1-\alpha_i\bar\alpha_j)\ne0$,
the function $g$ vanishes at the zeros of $q_1\cdots q_n$ with at least the
same multiplicities, so $g=q_1\cdots q_n\,h$ and $f=\eta\,p_1\cdots p_n\,h$
with $h$ entire.

For an entire function $\varphi=\sum c_kz^k$ and $\alpha\in\D$, the coefficients
of $(z-\bar\alpha\rho)\varphi$ and $(\rho-\alpha z)\varphi$ are
$c_{k-1}-\bar\alpha\rho c_k$ and $\rho c_k-\alpha c_{k-1}$.  The cross terms
of their squared moduli coincide, as in Lemma~\ref{lem:identity}, and we get
\[
    M_2\bigl(r,(z-\bar\alpha\rho)\varphi\bigr)^2-M_2\bigl(r,(\rho-\alpha z)\varphi\bigr)^2
    =(1-\abs\alpha^2)\,(r^2-\rho^2)\,M_2(r,\varphi)^2 .
\]
Replacing the factors $p_m$ by $q_m$ one at a time and applying this identity
with $\varphi=h_m=q_1\cdots q_{m-1}\,p_{m+1}\cdots p_n\,h$, we obtain
\[
    M_2(\sqrt t,g)^2-M_2(\sqrt t,f)^2=(t-\rho^2)\,K(t),
\]
where
\[
    K(t)=\sum_{m=1}^n(1-\abs{\alpha_m}^2)\,M_2(\sqrt t,h_m)^2 .
\]
Here $K$ is a non-zero power series with non-negative coefficients, so it is
positive and non-decreasing on $(0,\infty)$.  Moreover, $\int_0^\infty
(1+t)K(t)W(t)\,dt<\infty$.  Indeed, $(t-\rho^2)K(t)$ is integrable against $W$
by~\eqref{eq:norm}, since $f,g\in F^2_W$.  Also $t-\rho^2\ge t/2$ for
$t\ge2\rho^2$, and $K$ is bounded on $[0,2\rho^2]$.  By~\eqref{eq:norm} and
$\int(t-s)W(t)\,dt=0$,
\[
    \begin{split}
    \norm g_W^2-\norm f_W^2
    &=\int_0^\infty(t-s)\bigl(K(t)-K(s)\bigr)W(t)\,dt\\
    &\qquad+(s-\rho^2)\int_0^\infty K(t)W(t)\,dt\ \ge\ 0 .
    \end{split}
\]
If equality holds, then $\rho=\sqrt s$, and $K(t)=K(s)$ on a set of positive
measure, so $K$ is constant and every $h_m$ is constant.  For $n\ge2$ this is
impossible, since $q_1$ divides $h_2$.  Hence $n=1$ and $h$ is constant, so
$g=q_1h$ is linear.  Conversely, if $\rho=\sqrt s$, $n=1$ and $g$ is linear,
then $h$ is constant, so $K$ is constant and both terms vanish.
\end{proof}

For $W(t)=e^{-t}$ the equality case recovers the extremal
pairs~\eqref{eq:extremal}.

Proposition~\ref{prop:blaschke} covers every pair that is tight on the whole
circle.  Let $f,g\in F^2_W$, $g\not\equiv0$, satisfy $\abs f\le\abs g$ for
$\abs z>\rho$ and $\abs f=\abs g$ on $\abs z=\rho$.  As in the proof of
Lemma~\ref{lem:circle}, $\omega=f/g$ is holomorphic on
$\{\abs z>\rho\}\cup\{\infty\}$ with $\abs\omega\le1$.  It has no poles on the
circle $\abs z=\rho$ either, since a pole there would make $\omega$ unbounded
in the exterior.  Hence $\psi(u)=\omega(\rho/u)$ is holomorphic in a
neighbourhood of $\overline\D$, and $\abs\psi=1$ on $\partial\D$.  Such a
function is a unimodular constant or a finite Blaschke product.  To see this,
let $B$ be the Blaschke product formed from the finitely many zeros of $\psi$
in $\D$, and apply the maximum principle to $\psi/B$ and $B/\psi$; this shows
that $\psi/B$ is a unimodular constant.  If $\psi$ is constant, then
$\norm f_W=\norm g_W$.  Otherwise $\psi$ is a finite Blaschke product of degree
$n\ge1$, $f(z)=\psi(\rho/z)\,g(z)$, and Proposition~\ref{prop:blaschke}
applies.
Consequently, a pair violating the principle at the radius $\sqrt s$ must
satisfy $\abs f<\abs g$ on part of the circle $\abs z=\sqrt s$.

For some weights, however, the criterion does fail at $R^2=s$: this happens,
for example, for weights concentrated near a circle.

\begin{proposition}\label{prop:thin}
Let $0<d\le41/100$ and $W=(2d)^{-1}\mathbf 1_{[1-d,1+d]}$, so that $m_0=m_1=1$.
No Borel probability measure $\nu$ on $[1,\infty)$
satisfies~\eqref{eq:moment-conditions} with $R=1$.
\end{proposition}

\begin{proof}
Suppose that $\nu$ does.  Since $\int t^{k-1}\,d\nu\le m_k\le(1+d)^k$ for all
$k$, the measure $\nu$ is supported on $[1,1+d]$.  The case $k=2$ gives
$\int t\,d\nu\le m_2=1+d^2/3$.  By convexity,
$t^k\le1+(t-1)\bigl((1+d)^k-1\bigr)/d$ on $[1,1+d]$, and hence
\[
    m_k\le\int t^k\,d\nu\le1+\tfrac d3\bigl((1+d)^k-1\bigr).
\]
Take $k=7$.  Then
$m_7=\bigl((1+d)^8-(1-d)^8\bigr)/(16d)=1+7d^2+7d^4+d^6$, and expanding gives
\[
    m_7-1-\tfrac d3\bigl((1+d)^7-1\bigr)=\tfrac{d^2}3\,P(d),
\]
where $P(d)=14-21d-14d^2-35d^3-18d^4-7d^5-d^6$.
The polynomial $P$ is decreasing on $(0,\infty)$, and
$P(41/100)=29878575059\cdot10^{-12}>0$.  Hence
$m_7>1+\frac d3\bigl((1+d)^7-1\bigr)$, a contradiction.
\end{proof}

For this weight, $\E\,(T-1)_+=d/4$ and $\Var T=d^2/3$, so
Proposition~\ref{prop:tail} applies exactly when $d \ge3 /4$.  As $d\to0$, the
weight tends to the point mass at $t=1$.  In that limit, the radius $1$ is
trivially admissible and $\nu=\delta_1$ satisfies~\eqref{eq:moment-conditions}.
So the criterion works for $d\ge3/4$ and in the limit $d=0$, but fails for
$0<d\le0.41$.  This suggests that the failure is a weakness of the method: the
circle estimate~\eqref{eq:circle} is averaged against a single measure, and for
such weights this loses too much.  We do not know whether $\mathcal{K}_W<1$ for any
$0<d<3/4$.  By Proposition~\ref{prop:blaschke}, a counterexample cannot be tight
on the unit circle.

For the Gamma weights of Corollary~\ref{cor:gamma} with $\kappa>\kappa_1$, the
range of $\kappa$ for which a suitable measure exists is unknown.  By scaling
we take $\alpha=1$, so that $m_k=\kappa(\kappa+1)\cdots(\kappa+k-1)$ and
$s=\kappa$.  Non-existence can be proved by duality.  If $\sigma_k,\tau_k\ge0$
for $1\le k\le K$ and
\begin{equation}\label{eq:dual-cert}
    \sum_{k=1}^K \sigma_k\frac{t^k}{m_k}-\sum_{k=1}^K \tau_k\frac{s\,t^{k-1}}{m_k}
    \ <\ \sum_{k=1}^K \sigma_k-\sum_{k=1}^K \tau_k
    \qquad\text{for all $t\ge s$},
\end{equation}
then no suitable $\nu$ exists: multiplying the conditions
\eqref{eq:moment-conditions} by $\sigma_k/m_k$ and $\tau_k/m_k$, subtracting
and summing gives the reverse of~\eqref{eq:dual-cert} integrated against $\nu$.

\begin{proposition}\label{prop:gamma50}
Let $W(t)=t^{49}e^{-t}/\Gamma(50)$, the weight of
Corollary~\ref{cor:gamma} with $\alpha=1$ and $\beta=49$.  No Borel
probability measure $\nu$ on $[50,\infty)$ satisfies~\eqref{eq:moment-conditions}
with $R^2=50$.
\end{proposition}

\begin{proof}
Here $\kappa=s=50$.  Take $K=30$, $\sigma_7=8252$, $\sigma_{28}=260$,
$\tau_2=90224$, $\tau_{19}=963$, $\tau_{30}=301$, and all other multipliers
zero; then $\sum\sigma_k-\sum\tau_k=-82976$.  Put $u=t/50$ and
$M_k=m_k/50^k=\prod_{j=0}^{k-1}(1+j/50)$.  Then \eqref{eq:dual-cert} for
$t\ge50$ is equivalent to $H(u)>0$ for $u\ge1$, where
\[
    H(u)=-82976+\frac{90224\,u}{M_2}-\frac{8252\,u^7}{M_7}
    +\frac{963\,u^{18}}{M_{19}}-\frac{260\,u^{28}}{M_{28}}
    +\frac{301\,u^{29}}{M_{30}}
\]
is a polynomial of degree $29$ with rational coefficients.  All coefficients
of $H(2+x)$ are positive, so $H>0$ on $[2,\infty)$.  For each of the intervals
$[a,b]=[1,\frac54]$, $[\frac54,\frac{11}8]$, $[\frac{11}8,\frac32]$ and
$[\frac32,2]$, all coefficients of the polynomial $(1+x)^{29}\,H\bigl((a+bx)/(1+x)\bigr)$
are positive, so $H>0$ on $[a,b)$, since $x\mapsto(a+bx)/(1+x)$ maps
$[0,\infty)$ onto $[a,b)$.  These are finitely many sign checks in exact
rational arithmetic; an independent verifier is archived
at~\cite{WikstromZenodo}.  (Numerically, the minimum of $H$ on $[1,\infty)$ is
$H(1)\approx8.23$.)
\end{proof}

With the same method (linear programming for the multipliers, then an exact
proof of positivity by Descartes' rule of signs with bisection) we have found
certificates with $K\le150$ for $\kappa=43.7$, $44$, $45$, $60$, $70$, $100$,
$200$, $300$, $500$, $1000$, $3000$ and $10000$.  They are archived, with an
exact verifier, at~\cite{WikstromZenodo}.  The discretized linear programs
truncated at $k\le K$ change from feasible to infeasible near
$\kappa\approx43.6$, for every $K$ between $40$ and $600$ that we tried.  Feasibility of a
discretized, truncated problem proves nothing about the full problem, and we
have no proof that $43.6$ is the true threshold.  Nor do we know whether
non-existence holds for every $\kappa\ge43.7$.

\begin{question}
Is $\mathcal{K}_W=\sqrt{m_1/m_0}$ for the weights
$\abs z^{2\beta}e^{-\alpha\abs z^2}$ with $\beta>\kappa_1-1$, and for the
weights of Proposition~\ref{prop:thin}?
\end{question}

\subsection{Other exponents}
For the spaces $F^p_\alpha$ with $1\le p<\infty$, $p\ne2$, the Korenblum
radius remains unknown; see \cite[Question~4.2]{WeeLe2020}.  (For $0<p<1$ the
principle fails \cite{HuLou2019}.)  For even exponents $p=2n$ the
map $f\mapsto f^n$ reduces the problem to $F^2_{n\alpha}$.  In the
normalization of \cite{WeeLe2020}, where the weight is
$e^{-p\alpha\abs{z}^2/2}$, Corollary~\ref{cor:alpha} gives the lower bound
$(n\alpha)^{-1/2}$.  The pair $f\equiv c$, $g(z)=z$ gives the upper bound
$(n!)^{1/(2n)}(n\alpha)^{-1/2}$ \cite{WeeLe2020}.  Our method does not
apply directly when $p\neq2$, since the norm is then no longer a moment
functional of the squared Taylor coefficients.

\subsection{The Bergman space}
Theorem~\ref{thm:decreasing} applies to $W=\mathbf 1_{[0,1]}$, for which
$\norm{f}_W$ is the norm of $A^2(\D)$ normalized so that $\norm{1}=1$, and
gives $\mathcal{K}_W=\sqrt{m_1/m_0}=1/\sqrt2$.  This exceeds the upper bound
$c_2<0.6779$ of \cite{Wang2008}, but there is no contradiction.  For a weight
with compact support, admissibility for $F^2_W$ as defined in
Section~\ref{sec:prelim} concerns entire functions and requires
$\abs{f}\le\abs{g}$ on the whole exterior $\{\abs{z}>R\}$, including the part
outside the support of $W$, where the norm does not see the functions.
Korenblum's problem only asks for domination on the annulus
$\{c<\abs{z}<1\}$, which is a much weaker hypothesis.


\end{document}